\documentclass{article}
\usepackage{amsmath}
\usepackage{amssymb}
\usepackage{fullpage}
\usepackage{algpseudocode}
\usepackage{graphics}
\usepackage{booktabs}
\usepackage{cite}
\usepackage{bbm}
\usepackage{tikz}
\usepackage{bm}
\usepackage{appendix}
\usepackage{hyperref}
\usepackage[margin=3cm]{geometry}

\usepackage{algorithm2e}

\usepackage{graphicx}
\usepackage{subfig}

\usepackage{multirow}
\usepackage{rotating}
\usepackage{verbatim}
\usepackage{optidef}
\usepackage{cleveref}
\usepackage{array}
\usepackage{tabularx}
\usepackage{arydshln}
\usepackage{mathtools}
\usepackage{easybmat}
\usetikzlibrary{positioning}

\usepackage{enumerate}
\usepackage{enumitem}
\usepackage{rotating}
\usepackage{pdflscape}

\usepackage{tikz,forest} 
\usetikzlibrary{arrows.meta}
\forestset{
    .style={
        for tree={
            base=bottom,
            child anchor=north,
            align=center,
            s sep+=1cm,
    straight edge/.style={
        edge path={\noexpand\path[\forestoption{edge},thick,-{Latex}] 
        (!u.parent anchor) -- (.child anchor);}
    },
    if n children={0}
        {tier=word, draw, thick, rectangle}
        {draw, diamond, thick, aspect=2},
    if n=1{%
        edge path={\noexpand\path[\forestoption{edge},thick,-{Latex}] 
        (!u.parent anchor) -| (.child anchor) node[pos=.2, above] {Y};}
        }{
        edge path={\noexpand\path[\forestoption{edge},thick,-{Latex}] 
        (!u.parent anchor) -| (.child anchor) node[pos=.2, above] {N};}
        }
        }
    }
}

\DeclarePairedDelimiterX{\Set}[1]\{\}{\setargs{#1}}
\NewDocumentCommand{\setargs}{>{\SplitArgument{1}{|}}m}
{\setargsaux#1}
\NewDocumentCommand{\setargsaux}{mm}
{\IfNoValueTF{#2}{#1}{\nonscript\,#1\nonscript\;\delimsize\vert\allowbreak \nonscript\:\mathopen{}#2\nonscript\,}}

\newcommand*\xbar[1]{%
  \hbox{%
    \vbox{%
      \hrule height 0.5pt 
      \kern0.5ex
      \hbox{%
        \kern-0.1em
        \ensuremath{#1}%
        \kern-0.1em
      }%
    }%
  }%
} 

\usepackage{tikz,pgfplots,xpatch}
\usetikzlibrary{shapes.multipart,backgrounds,shapes,fit}
\usetikzlibrary{arrows,decorations.markings,shapes.arrows,arrows.meta}

\newcommand{\E}{\mathbb{E}}

\newcommand{\R}{\mathbb{R}}

\renewcommand{\P}{\mathbb{P}}

\DeclareMathOperator{\OPT}{\operatorname{OPT}}
\DeclareMathOperator{\ROPT}{\operatorname{RPOPT}}
\DeclareMathOperator{\EOPT}{\operatorname{EBOPT}}

\definecolor{mypink3}{cmyk}{0, 0.7808, 0.4429, 0.1412}

\newcommand{\norm}[1]{\left\lVert#1\right\rVert}

\newcommand{\ev}[0]{\mathbf{e}}

\newcommand{\xv}[0]{\mathbf{x}}
\newcommand{\yv}[0]{\mathbf{y}}

\newcommand{\bv}[0]{\mathbf{b}}
\newcommand{\gv}[0]{\mathbf{g}}
\newcommand{\sv}[0]{\mathbf{s}}

\newcommand{\rv}[0]{\mathbf{r}}

\newcommand{\tv}[0]{\mathbf{t}}

\newtheorem{thm}{Theorem}[section]

\usepackage{booktabs}

\newtheorem{lem}[thm]{Lemma}

\newtheorem{prop}[thm]{Proposition}

\newtheorem{rem}[thm]{Remark} 
\newenvironment{proof}[1][]
   {\textsc{Proof#1: }}
   {\hfill$\Box$\smallskip}

\newcommand{\iprod}[2]{\left\langle {#1}, {#2} \right\rangle}

\newcounter{algorithmctr}[section]
\renewcommand{\thealgorithmctr}{\thesection.\arabic{algorithmctr}}
    {\refstepcounter{algorithmctr}\begin{list}{}{%
        \setlength{\rightmargin}{.05\linewidth}%
        \setlength{\leftmargin}{.05\linewidth}}%
        \item[]{\setlength{\parskip}{0ex}\bigskip\par%
         \nopagebreak%
         \underline{{\bf Algorithm \thealgorithmctr.} \emph{#1.}}}}%
    {{\setlength{\parskip}{-1ex}\nopagebreak\smallskip\par} \end{list}}

\makeatletter
\newsavebox\myboxA
\newsavebox\myboxB
\newlength\mylenA

\makeatother

\title{Random projection with linear equalities and inequalities}
\author{Santanu S. Dey and Jingye Xu}

\begin{document}

\maketitle
\begin{abstract}
Random projection is a preprocessing technique for optimization that reduces the number of linear equality constraints by replacing them with random linear combinations while approximately preserving the optimal value. In this paper, we emphasize the role of suprema of random processes, and in particular generic chaining, in explaining this phenomenon. By combining Lagrangian duality with generic chaining, we obtain shorter proofs and sharper guarantees that apply beyond the typical polyhedral settings based on the Johnson--Lindenstrauss lemma. More importantly, this perspective allows us to extend random projection to problems with linear inequality constraints. Although randomly projected inequality constraints generally define neither a relaxation nor a restriction of the original feasible region, we show that random projection can nevertheless approximately preserve the dual bound. The key is a new construction for randomly projecting dual certificates that preserves dual feasibility and differs from existing constructions motivated by the Johnson--Lindenstrauss lemma.

\end{abstract}

\section{Introduction}
\label{sec:introduction}

Consider a convex program
\begin{equation}
    \label{prob_eq}
    \begin{aligned}
        \OPT_{=} := \inf_{\xv} \quad & f(\xv) \\
        \text{s.t.} \quad
        & A\xv = \bv, \\
        & \xv \in S,
    \end{aligned}
\end{equation}
where $S \subseteq \R^{n}$ is a closed convex set,
$f:\R^n\to\R$ is a convex function, and
$A \in \R^{m\times n}$. 

To apply random projection to \eqref{prob_eq}, we typically sample a matrix
$G\in\R^{k\times m}$ whose entries are independently distributed as Gaussian 
$N(0,1/k)$, where $k\ll m$. We then replace the original system
$A\xv=\bv$ with $k$ random linear combinations of its constraints:
\begin{equation}
    \label{prob_eq_proj}
    \begin{aligned}
        \ROPT_{=}(k) := \inf_{\xv} \quad & f(\xv) \\
        \text{s.t.} \quad
        & GA\xv = G\bv, \\
        & \xv \in S.
    \end{aligned}
\end{equation}
Since $k\ll m$, the projected formulation contains substantially
fewer equality constraints than the original problem. This reduction
can decrease the costs of model construction, matrix factorization,
and linear-system solution, which are often major computational
bottlenecks in large-scale optimization. Moreover, every feasible
solution of \eqref{prob_eq} is feasible for \eqref{prob_eq_proj}.
Therefore, the projected problem is a relaxation of the original
problem, and hence
$\ROPT_{=}(k) \leq \OPT_{=}.
$

The central 
question is whether the projected relaxation
remains sufficiently tight. Under suitable assumptions and with $k$
chosen according to the geometry of $S$ and $A$, one seeks a
high-probability guarantee of the form
$
    \OPT_{=} - \epsilon(k)
    \leq
    \ROPT_{=}(k),
$
where $\epsilon(k)>0$ decreases as $k$ increases. 
For example, Vu et al.~\cite{vu2018random} consider a linear program
over a scaled simplex under a normalization assumption on the columns
of $A$ and $\bv$. Let $\lambda_\star$ denote a minimum-norm optimal dual solution
associated with the equality constraints $A\xv=\bv$. They show that
\[
    \epsilon(k)
    =
    O\!\left(
        \|\lambda_\star\|_2
        \sqrt{\frac{\log(n)}{k}}
    \right)
\]
with high probability.
Extensions of random-projection techniques to quadratic programs are
studied in~\cite{dambrosio2020random}, while extensions to symmetric
conic programs, including linear, second-order cone, and semidefinite
programs, are developed in~\cite{liberti2021random}. We also note that, in the random-projection literature, it is typically assumed that the feasible region is contained in a suitably scaled simplex, and existing proofs often exploit this assumption explicitly.

Random projection is also closely related to another dimensionality-reduction
technique, known as \emph{random embedding}, in which the decision variables
are restricted to a randomly generated low-dimensional linear subspace.
For example, when random embedding is applied to \eqref{prob_eq}, we sample
a matrix
$
    P \in \R^{n \times k},
$
whose entries are independently distributed as $N(0,1/k)$, where $k \ll n$,
and parameterize the decision vector as
$
    \xv = P\yv.
$
This gives the embedded problem:
\begin{equation}
    \label{prob_eq_ebd}
    \begin{aligned}
        \EOPT_{=}(k) := \inf_{\yv} \quad & f(P\yv) \\
        \text{s.t.} \quad
        & AP\yv = \bv, \\
        & P\yv \in S.
    \end{aligned}
\end{equation}
Under suitable assumptions and with the embedding dimension $k$ chosen
appropriately, one can obtain a similar two-sided approximation guarantee of
the form
$
    \OPT_{=}
    \leq
    \EOPT_{=}(k)
    \leq
    \OPT_{=}+\epsilon(k).
$
The close connection between random projection and random embedding can be
understood through Lagrangian duality. In particular, randomly aggregating the
primal equality constraints is equivalent to restricting the corresponding
dual variables to a randomly generated low-dimensional subspace.  As a result, the two techniques often admit closely related
approximation guarantees and proof techniques, although they act on different sides of the
primal--dual pair. For a broader discussion of random embedding methods in
optimization, we refer to~\cite{cartis2023global}.

In this paper, we emphasize the role of the theory of suprema of random processes, particularly \textit{generic chaining} \cite{talagrand2005generic}, in the analysis of dimension-reduction techniques for optimization.  
Many existing guarantees for these techniques are, 
either directly or indirectly, rooted in the study of suprema of random processes. 
However, this connection is often left implicit, and random-process tools are rarely 
invoked directly in the corresponding analyses.
For instance, \cite{cartis2023global} relies on statistical dimension, which is closely related to Gaussian width and generic chaining complexity, while \cite{liberti2021random} builds on the Johnson--Lindenstrauss lemma, which can be viewed as a special case of Gordon-type random embedding results rooted in generic chaining. These connections suggest that random projection methods in optimization can be naturally understood through the theory of generic chaining.

Using convex duality, we show that the projection error $\epsilon(k)$ can be expressed as the supremum of an appropriate random process, which can then be controlled using generic chaining. This viewpoint leads to shorter proofs and sharper guarantees. In particular, we generalize the result of \cite{vu2018random} and show that
$
    \OPT_{=} - \epsilon(k)
    \leq
    \ROPT_{=}(k)
    \leq
    \OPT_{=},
$
where
\begin{align*}
\epsilon(k)
\lesssim
\norm{\lambda_\star}_2
\left[
\frac{
\gamma_2(AS,|\cdot|_2)
+\operatorname{rad}(AS)
+\norm{\bv}_2
}{\sqrt{k}}
\right].
\end{align*}
Here, \(\gamma_p(\cdot)\) denotes the \(p\)-chaining functional, \(\lambda_\star\) is an optimal dual solution associated with \(A\xv=\bv\), and \(\operatorname{rad}(\cdot)\) denotes the \(\ell_2\)-radius of a set. We formally introduce the \(p\)-chaining functional in Section~\ref{sec:chaining}. Readers unfamiliar with generic chaining may simply view \(\gamma_p(AS,\|\cdot\|_2)\) as a geometric complexity measure of \(AS\), analogous to covering numbers or \(\epsilon\)-nets but more sensitive to the geometry of the underlying set.

This new viewpoint leads to several improvements. First, we extend existing random-projection results for linear objectives \cite{vu2018random,liberti2021random} to general convex objective functions. More fundamentally, the convex-duality perspective reveals that the complexity of the objective function itself plays little role in the projection error; instead, the behavior of the associated random process is governed primarily by the geometry of the feasible region.
When \(S\) is a scaled simplex, our bound matches that of \cite{vu2018random}. More importantly, our analysis removes the need for the feasible region to admit a polyhedral description and applies directly to general convex sets, in some cases yielding strictly sharper bounds. For example, suppose \(AS\) is the intersection of the unit \(\ell_2\)-ball with the positive orthant. A simplex-based argument requires including this set into a simplex scaled by a factor of order \(\sqrt{n}\). Under suitable normalization conditions, the analysis of \cite{vu2018random} then yields
$
    \epsilon(k)
    =
    O\!\left(
        \|\lambda_\star\|_2
        \sqrt{\frac{n\log n}{k}}
    \right).
$
In contrast, our chaining-based analysis gives
$
    \epsilon(k)
    =
    O\!\left(
        \|\lambda_\star\|_2
        \sqrt{\frac{n}{k}}
    \right),
$
since
$
    \gamma_2(AS,\|\cdot\|_2)
    =
    \Theta(\sqrt{n})
$
for this set. Thus, the chaining-based viewpoint not only extends the analysis beyond polyhedral feasible regions, but can also exploit their intrinsic geometry to obtain sharper projection guarantees.

On the other hand, we illustrate the power of the random-process perspective by
extending existing results for linear equality constraints to the setting of linear
inequality constraints. Consider the analogous convex program
\begin{equation}
    \label{prob_ineq}
    \begin{aligned}
        \OPT_{\leq} := \inf_{\xv} \quad & f(\xv) \\
        \text{s.t.} \quad
        & A\xv \leq \bv, \\
        & \xv \in S,
    \end{aligned}
\end{equation}
and its randomly projected counterpart
\begin{equation}
    \label{prob_ineq_proj}
    \begin{aligned}
        \ROPT_{\leq}(k) := \inf_{\xv} \quad & f(\xv) \\
        \text{s.t.} \quad
        & GA\xv \leq G\bv, \\
        & \xv \in S.
    \end{aligned}
\end{equation}
A fundamental difficulty arises immediately. Unlike in the equality-constrained
setting, an arbitrary random matrix \(G\) does not preserve feasibility. Indeed,
$
    A\xv \leq \bv
$
does not in general imply
$
    GA\xv \leq G\bv
$
when \(G\) contains negative entries. Consequently, the feasible regions of
\eqref{prob_ineq} and \eqref{prob_ineq_proj} are generally incomparable:
the projected problem is neither necessarily a relaxation nor a restriction of the
original problem. At first sight, there is therefore no reason to expect any
deterministic ordering between \(\ROPT_{\leq}(k)\) and \(\OPT_{\leq}\), and it is
unclear whether the classical random-projection framework can be meaningfully
extended to inequality-constrained optimization. 
 
For typical choice of $G$, that is, $G$ with each entry i.i.d distributed as $N(0,1/k)$, while random projection can not preserve primal feasible and therefore it can not preserve primal bound. We will show that it can preseves dual bound.
In particular, with high probability, $
    \OPT_{\leq}-\epsilon(k)
    \leq
    \ROPT_{\leq}(k),
$
where
\begin{align*}
\epsilon(k)
\lesssim
\left(\norm{\lambda_\star}_2+\norm{\lambda_\star}_1\right)
\left[
\frac{
\gamma_2(AS,\|\cdot\|_2)
+\norm{\bv}_2
}{\sqrt{k}}
+
\frac{
\gamma_1(AS,\|\cdot\|_2)
}{k}
\right].
\end{align*}

At first glance, the above result may appear counterintuitive. After all, the
projected problem is obtained by aggregating the inequality constraints with a
random matrix that may contain negative entries, which is neither a relaxation nor a restriction. How, then, can one still establish a meaningful comparison between
the projected and original problems?
In fact, the existing guarantee
$
\OPT_{=} - \epsilon(k)
    \leq
    \ROPT_{=}(k)
$
is rooted in the observation that random projection preserves the dual
structure, or more precisely, the KKT system. Once dual feasibility and
complementary slackness are satisfied, the KKT conditions do not distinguish between equality constraints and active inequality constraints.
Indeed, only the active inequality constraints are used to certify optimality, and they
effectively behave as equality constraints. This observation suggests that a
similar guarantee may also hold for inequality-constrained problems.
The main difficulty, however, is that random projection no longer preserves
dual feasibility in the inequality-constrained setting. To overcome this
difficulty, we exploit the sign-symmetric structure of the random matrix
\(G\) and introduce a new way of projecting the dual solution, which differs from existing approaches. See Section~\ref{subsec:random_proj_as_process} for details.

The two error bounds above have several interesting implications. Unlike the linear equality-constrained case, the error bound for inequality-constrained problems is sensitive to the complexity of the objective function. When the objective function is linear, the set \(AS\) appearing in the bounds can be replaced by its set of vertices \(V(AS)\), see Remark \ref{rem:lin_help} and Table \ref{tab:gamma_simplex_ball} . Although this replacement does not substantially improve the \(\gamma_2\)-functional, it can lead to an order-wise improvement in the \(\gamma_1\)-functional.
On the other hand, both chaining functionals satisfy a natural monotonicity property. Namely, if \(A\subseteq B\), then
$
\gamma_2(A,\norm{\cdot}_2)
\leq
\gamma_2(B,\norm{\cdot}_2)\text{ and }
\gamma_1(A,\norm{\cdot}_2)
\leq
\gamma_1(B,\norm{\cdot}_2).
$
Consequently, applying random projection to a larger feasible set may require a larger projection dimension to achieve the same error guarantee, even when the larger set admits a conceptually simple geometric description. 

\begin{table}[ht]
\centering
\caption{Orders of the chaining functionals for the Euclidean simplex,
the Euclidean ball, and their extreme points. Here
\(\Delta_n:=\operatorname{conv}\{0,e_1,\ldots,e_n\}\),
\(B_2^n:=\{x\in\mathbb{R}^n:\|x\|_2\leq1\}\), and
\(\asymp\) hides universal constants.}
\label{tab:gamma_simplex_ball}
\begin{tabular}{lccc}
\toprule
Set \(T\)
&
\(\gamma_2(T,\|\cdot\|_2)\)
&
\(\gamma_1(T,\|\cdot\|_2)\)
&
Reference
\\
\midrule

\(\displaystyle
\Delta_n
=
\operatorname{conv}\{0,e_1,\ldots,e_n\}
\)
&
\(\asymp \sqrt{\log n}\)
&
\(\asymp \sqrt{n}\)
&
\cite{talagrand2005generic} and Prop \ref{prop_gamma_1}
\\[1.5mm]

\(\displaystyle
V(\Delta_n)
=
\{0,e_1,\ldots,e_n\}
\)
&
\(\asymp \sqrt{\log n}\)
&
\(\asymp \log n\)
&
\cite{talagrand2005generic}
\\[1.5mm]

\(\displaystyle
B_2^n
=
\{x\in\mathbb{R}^n:\|x\|_2\leq1\}
\)
&
\(\asymp \sqrt{n}\)
&
\(\asymp n\)
&
\cite{talagrand2005generic,vershynin2018highdim}
\\[1.5mm]

\(\displaystyle
V(B_2^n)
=
S^{n-1}
\)
&
\(\asymp \sqrt{n}\)
&
\(\asymp n\)
&
\cite{talagrand2005generic,vershynin2018highdim}
\\
\bottomrule
\end{tabular}
\end{table}

The remainder of the paper is organized as follows. In
Section~\ref{sec:prelim}, we introduce the notation and background on random
processes, and explain how random projection can be formulated in terms of an
appropriate random process. In Section~\ref{sec:gurantee}, we derive
random-projection guarantees for equality- and inequality-constrained problems,
respectively. We conclude the paper with a brief discussion of the main results
and open questions.

\section{Preliminary results and notation}
\label{sec:prelim}

In this section, we introduce several notions and results from probability theory that will be used in later sections.

For a set \(T\subseteq\R^m\), define
$
    \operatorname{rad}(T)
    :=
    \sup_{\tv\in T}\|\tv\|_2
$
and
$
    \operatorname{diam}(T)
    :=
    \sup_{\tv,\sv\in T}\|\tv-\sv\|_2.
$
For a random variable \(X\), we use the standard Orlicz norms
$
\|X\|_{\psi_2}
:=
\inf\left\{
t>0:
\mathbb{E}\exp\left(\frac{X^2}{t^2}\right)\leq2
\right\}
$
and
$
\|X\|_{\psi_1}
:=
\inf\left\{
t>0:
\mathbb{E}\exp\left(\frac{|X|}{t}\right)\leq2
\right\}.
$ A random variable \(X\) is called \emph{sub-Gaussian} if
$
\|X\|_{\psi_2}<\infty,
$
and \emph{sub-exponential} if
$
\|X\|_{\psi_1}<\infty.
$
Throughout the paper, \(C,c>0\) denote universal constants whose values may
change from line to line. We write \(a\lesssim b\) if \(a\leq Cb\).

We also require several standard results from probability theory.

\begin{lem} \cite{vershynin2018highdim}
\label{lem:centering}
For every sub-Gaussian random variable \(X\),
$
    \|X-\mathbb{E}X\|_{\psi_2}
    \lesssim
    \|X\|_{\psi_2}.
$
\end{lem}

\begin{lem} \cite{vershynin2018highdim}
\label{lem:product}
Let \(X\) and \(Y\) be sub-Gaussian random variables, not necessarily
independent. Then \(XY\) is sub-exponential and
$
    \|XY\|_{\psi_1}
    \lesssim
    \|X\|_{\psi_2}\|Y\|_{\psi_2}.
$
\end{lem}

\begin{lem}\cite{vershynin2018highdim}
\label{lem:constant_psi_2}
For every deterministic scalar \(a\),
$
    \|a\|_{\psi_2}
    \lesssim
    |a| \text{ and }
    \|a\|_{\psi_1}
    \lesssim
    |a|.
$
\end{lem}



\subsection{Random processes and generic chaining}
\label{sec:chaining}

In this subsection, we provide a brief introduction to random processes and generic chaining. 
A random process is a collection of random variables
$
    \{X_t : t \in T\}
$
indexed by a set \(T\), and it is called centered if
$
    \E[X_t] = 0, \forall t \in T.
$
Throughout this paper, we assume that \(T\) is contained in a normed vector space equipped with a metric \(d(\cdot,\cdot)\). 
The main object of interest is the expected supremum
$
    \E \sup_{t \in T} X_t.
$
Controlling such a supremum is more difficult than controlling a single random variable, since one must account not only for the tail behavior of each \(X_t\), but also for the dependence across the collection \(\{X_t\}_{t\in T}\).

Generic chaining provides a systematic way to control this expected supremum by exploiting the geometric complexity of the index set \(T\) under the (semi)-metric induced by the increments $
    X_t-X_s, s,t\in T.
$ of the process.
For example, a process is called a \emph{sub-Gaussian process} with respect to a metric \(d\) if
$
    \P\left(
        |X_t-X_s|
        \geq u\,d(t,s)
    \right)
    \leq
    2\exp(-c u^2),
$
for all \(s,t\in T\) and \(u\geq 0\), where \(c>0\) is a universal constant. 
Equivalently, this condition can be written as
$
    \|X_t-X_s\|_{\psi_2}
    \lesssim
    d(t,s),\forall s,t\in T.
$
Thus, points that are close under \(d\) correspond to random variables whose increments are small with high probability. 

A first way to exploit this geometry is through covering numbers. 
For \(\varepsilon>0\), an \(\varepsilon\)-net of \(T\) with respect to \(d\) is a finite set
\(T_\varepsilon\subseteq T\) such that, for every \(t\in T\), there exists some
\(s\in T_\varepsilon\) satisfying
$
    d(t,s)\leq \varepsilon.
$
The covering number
$
    N(T,d,\varepsilon)
$
is the minimum cardinality of such an \(\varepsilon\)-net. 
At a fixed scale \(\varepsilon\), one may approximate each point \(t\in T\) by a nearby point in an \(\varepsilon\)-net and then apply a union bound over the finitely many points in the net. 
Repeating this argument over multiple scales leads to entropy-integral bounds. 
In particular, for a sub-Gaussian process, Dudley's inequality \cite{dudley1967sizes} gives
\[
    \E\sup_{t\in T}|X_t-X_{t_0}|
    \lesssim
    \int_0
    \sqrt{\log N(T,d,\varepsilon)}
    \,d\varepsilon,
\]
for any fixed \(t_0\in T\).
Although Dudley's inequality is useful in many settings, it can be suboptimal because the covering sets at different scales are chosen and analyzed separately.

Generic chaining refines this multiscale argument by approximating every point of \(T\) simultaneously across a sequence of increasingly fine resolutions. 
An \emph{admissible sequence} of \(T\) is a sequence of finite subsets
$
    \mathcal T_0,\mathcal T_1,\ldots \subseteq T
$
satisfying
$
    |\mathcal T_0|=1,
    |\mathcal T_n|
    \leq
    2^{2^n},n\geq 1.
$
For an admissible sequence \(\{\mathcal T_n\}_{n\geq 0}\), define
$
    d(t,\mathcal T_n)
    :=
    \inf_{s\in\mathcal T_n} d(t,s),
$
and let
$
    \pi_n(t)\in \mathcal T_n
$
be a closest point to \(t\), so that
$
    d(t,\pi_n(t))
    =
    d(t,\mathcal T_n).
$
If the closest point is not unique, one may choose one arbitrarily.
The basic idea of chaining is to represent \(X_t\) through successive approximations of \(t\). 
Formally, assuming some regularity condition so that \(\pi_n(t)\to t\) as $n \to \infty$, one may write
\[
    X_t-X_{\pi_0(t)}
    =
    \sum_{n\geq 0}
    \left(
        X_{\pi_{n+1}(t)}
        -
        X_{\pi_n(t)}
    \right).
\]
Since \(|\mathcal T_0|=1\), the point \(\pi_0(t)\) is independent of \(t\). 
If the process is centered, then
$
    \E X_{\pi_0(t)}=0.
$
Consequently,
\begin{align*}
    \E\sup_{t\in T} X_t
    &\leq
    \sum_{n\geq 0}
    \E\sup_{t\in T}
    \left(
        X_{\pi_{n+1}(t)}
        -
        X_{\pi_n(t)}
    \right). 
\end{align*}
The increment at level \(n\) is controlled by the distance between two successive approximations, while the admissibility condition ensures that the number of possible increments at level \(n\) does not grow too quickly. 
This balance between the size of the increments and the number of possible increments is the main idea behind generic chaining.
For \(p\geq 1\), the corresponding chaining functional is defined as
\[
    \gamma_p(T,d)
    :=
    \inf_{\{\mathcal T_n\}}
    \sup_{t\in T}
    \sum_{n\geq 0}
    2^{n/p} d(t,\mathcal T_n),
\]
where the infimum is taken over all admissible sequences
\(\{\mathcal T_n\}_{n\geq 0}\). 
The functional \(\gamma_p(T,d)\) measures the geometric complexity of \(T\) simultaneously across all scales. 
In particular, \(\gamma_2\) is the natural complexity measure for processes with sub-Gaussian increments.

Indeed, if
\[
    \|X_t-X_s\|_{\psi_2}
    \lesssim
    d(t,s),
    \qquad s,t\in T,
\]
then generic chaining yields
\[
    \E\sup_{t\in T}
    |X_t-X_{t_0}|
    \lesssim
    \gamma_2(T,d)
\]

More generally, random processes whose increments can be represented as sums of independent centered sub-exponential random variables naturally give rise to mixed-tail behavior. Given two semimetrics \(d_2\) and \(d_1\) on \(T\), we say that a random process \(\{X_t\}_{t\in T}\) satisfies a \emph{mixed-tail increment condition} if
\begin{equation}
\label{eq:mixed_tail}
\P\left(
|X_t-X_s|
\geq
C\left[
\sqrt{u},d_2(t,s)
+
u,d_1(t,s)
\right]
\right)
\leq
2e^{-u},
\qquad
\forall s,t\in T,\ u\geq 1,
\end{equation}
for some universal constant \(C>0\).

A standard mechanism leading to \eqref{eq:mixed_tail} is Bernstein's inequality for sums of independent sub-exponential random variables. We first recall the following form of Bernstein's inequality.

\begin{lem}[Bernstein's inequality for sub-exponential random variables]
\label{lem:bernstein_subexponential}
Let \(X_1,\ldots,X_k\) be independent centered random variables satisfying
$
    \|X_i\|_{\psi_1}\le K,
    \qquad i=1,\ldots,k.
$
Then, for every \(u\ge 0\), with probability at least \(1-Ce^{-cu}\),
$
    \left|
    \frac{1}{k}\sum_{i=1}^k X_i
    \right|
    \le
    K
    \left(
    \sqrt{\frac{u}{k}}
    +
    \frac{u}{k}
    \right),
$
where \(C,c>0\) are universal constants.
\end{lem}

As a direct consequence, sums of independent sub-exponential increments satisfy a mixed-tail increment condition.

\begin{lem}
\label{lem:subexp_mixed_tail}
Let \(\{X_t\}_{t\in T}\) be a random process indexed by
\(T\subseteq\R^n\). Suppose that, for every \(s,t\in T\),
$
    X_t-X_s
    =
    \sum_{i=1}^k Y_i(s,t),
$
where \(Y_1(s,t),\ldots,Y_k(s,t)\) are independent centered random variables
satisfying
$
    \|Y_i(s,t)\|_{\psi_1}
    \lesssim
    \frac{K}{k}\|t-s\|_2,
    i=1,\ldots,k.
$
Then, for every \(u\geq 1\),
$
\P\left(
    |X_t-X_s|
    >
    CK\|t-s\|_2
    \left[
        \sqrt{\frac{u}{k}}
        +
        \frac{u}{k}
    \right]
\right)
\leq
2e^{-u},
$
where \(C>0\) is a universal constant. Consequently,
\(\{X_t\}_{t\in T}\) satisfies the mixed-tail increment condition
\eqref{eq:mixed_tail} with
$
    d_2(t,s)
    =
    \frac{K}{\sqrt{k}}\|t-s\|_2
    \text{ and }
    d_1(t,s)
    =
    \frac{K}{k}\|t-s\|_2.
$

\end{lem}

Once a process satisfies \eqref{eq:mixed_tail}, its supremum can be controlled by generic chaining through the two corresponding geometric complexity measures
\(\gamma_2(T,d_2)\) and \(\gamma_1(T,d_1)\). We use the following mixed-tail generic chaining bound of \cite{dirksen2015tail}.

\begin{lem}[Generic chaining for mixed-tail processes
{\cite{dirksen2015tail}}]
\label{lem:mixed_tail_chaining}
Let \(\{X_t\}_{t\in T}\) be a random process satisfying the mixed-tail
increment condition \eqref{eq:mixed_tail} with respect to semimetrics
\(d_2\) and \(d_1\). Then, for every fixed \(t_0\in T\),

$$
    \E\sup_{t\in T}
    |X_t-X_{t_0}|
    \lesssim
    \gamma_2(T,d_2)
    +
    \gamma_1(T,d_1).
$$

Moreover, there exists a universal constant \(c>0\) such that, for every
\(u\geq 1\), with probability at least \(1-e^{-cu}\),

$$
\sup_{t\in T}
|X_t-X_{t_0}|
\lesssim
\gamma_2(T,d_2)
+
\gamma_1(T,d_1)
+
\sqrt{u}\,\operatorname{diam}(T,d_2)
+
u\,\operatorname{diam}(T,d_1).
$$

\end{lem}

We will also need a related bound for Gaussian linear processes. Unlike the preceding result, this estimate involves only the \(\gamma_2\)-functional, reflecting the purely sub-Gaussian nature of Gaussian increments. The following lemma is a direct consequence of standard Gaussian comparison together with Gaussian concentration. We include the proof in Appendix \ref{app_proofs}.

\begin{lem}
\label{lem:gaussian_linear_process}
Let \(T\subseteq\mathbb{R}^m\) be bounded and let
\(g\sim N(\mathbf 0,I_m)\). Let \(P:\mathbb{R}^m\to\mathbb{R}^m\) be a
linear contraction, i.e.,
$
    \|P\xv\|_2\leq \|\xv\|_2
    \qquad
    \text{for all } \xv\in\mathbb{R}^m.
$
Then, for every \(u\geq 1\), with probability at least
\(1-Ce^{-cu}\),

$$
    \sup_{\tv\in T}
    \left|
    \langle Pg,\tv\rangle
    \right|
    \lesssim
    \gamma_2(T,\|\cdot\|_2)
    +
    \operatorname{rad}(T)\sqrt{u},
$$

where \(c,C>0\) are universal constants.
\end{lem}

\subsection{Random projection as a random process}
\label{subsec:random_proj_as_process}
In this subsection, we formulate the random projection error as the supremum
of a stochastic process.
We begin with the classical random projection approach for equality-constrained
problems, whose analysis is based on the Johnson--Lindenstrauss (JL) lemma \cite{matouvsek2013lecture}.
We will show how generic chaining generalizes arguments based on the JL lemma
and can be applied to difficult cases, i.e., non-polyhedral sets or inequality
cases.

The JL lemma is a fundamental result in dimension reduction, stating that for any finite set $H\subseteq\mathbb{R}^n$ and any $\epsilon\in(0,1)$, there exists a projection dimension $k=O(\epsilon^{-2}\log |H|)$ such that an appropriately scaled Gaussian random matrix $G\in\mathbb{R}^{k\times n}$ satisfies, with high probability,
$
(1-\epsilon)\|\xv\|_2
\le
\|G\xv\|_2
\le
(1+\epsilon)\|\xv\|_2,
\forall\,\xv\in H.
$
Since the Euclidean inner product can be recovered from the norm through the polarization identity, the JL lemma further implies approximate preservation of inner products. In particular, if all vectors in the finite sets $H_1,H_2\subseteq\mathbb{R}^n$ are normalized, then, with high probability,
$
\left|
\iprod{G\xv}{G\yv}
-
\iprod{\xv}{\yv}
\right|
\le
\epsilon,
\forall\,\xv\in H_1,\ \yv\in H_2.
$

At the heart of random projection methods is the approximate preservation of bilinear forms.
Under suitable regularity conditions, the optimality of a mathematical program is certified by its Karush--Kuhn--Tucker (KKT) system, whose defining relations consist of bilinear equalities between primal and dual variables. Consequently, random projection approximately preserves these bilinear relations and hence the associated optimality certificates.

We now illustrate this principle for the equality-constrained problem \eqref{prob_eq}. Let $\lambda_{\star}$ be an optimal dual solution corresponding to the equality constraint $A\xv=\bv$. Assuming strong duality, we have
\begin{align*}
\OPT_{=}
=
\inf_{\xv}
\quad &
f(\xv)
+
\iprod{\lambda_{\star}}{A\xv-\bv}
\\
\text{s.t.}\quad
&
\xv\in S.
\end{align*}
For the projected problem, choosing $G\lambda_{\star}$ as a dual certificate and applying weak duality yields
\begin{align*}
\ROPT_{=}
\ge
\inf_{\xv}
\quad &
f(\xv)
+
\iprod{G\lambda_{\star}}{GA\xv-G\bv}
\\
\text{s.t.}\quad
&
\xv\in S.
\end{align*}
Since the JL lemma approximately preserves bilinear forms, with high probability,
$
f(\xv)
+
\iprod{G\lambda_{\star}}{GA\xv-G\bv}
\approx
f(\xv)
+
\iprod{\lambda_{\star}}{A\xv-\bv},
\forall\,\xv\in S.
$
Combined with the trivial inequality $\ROPT_{=}\le\OPT_{=}$, this yields the approximation
$
\ROPT_{=}\approx\OPT_{=}.
$
Accordingly, the projection error is characterized by the maximum error between the projected and original bilinear forms:
$$
\epsilon(k)
\le
\sup_{\xv\in S}
\iprod{\lambda_{\star}}{A\xv-\bv} -
\iprod{G\lambda_{\star}}{GA\xv-G\bv}.
$$
where the right-hand side naturally fits into the framework of random processes and can be analyzed using generic chaining.

When extending this idea to inequality-constrained problems, the main difficulty is that a random projection does not preserve dual feasibility. Specifically,
$
\lambda\ge0
\;\not\Rightarrow\;
G\lambda\ge0.
$
Our key observation is that the JL lemma itself is not essential for preserving
bilinear structures. The only property required by the preceding argument is
the unbiasedness identity
$
\mathbb{E}(G^{\top}G)=I_n,
$
which implies
$
\mathbb{E}\!\left[
\iprod{G\lambda_{\star}}{GA\xv-G\bv}
\right]
=
\mathbb{E}\!\left[
\iprod{G^{\top}G\lambda_{\star}}{A\xv-\bv}
\right]
=
\iprod{\lambda_{\star}}{A\xv-\bv}.
$
Therefore, there is no need to project the dual multiplier as
$G\lambda_{\star}$, as suggested by the classical JL framework. Instead, we
project the dual multiplier as $2G^{+}\lambda_{\star}$, where $G^{+}$ denotes
the entrywise positive part of $G$. This projection preserves dual feasibility
while simultaneously satisfying
$
\mathbb{E}\!\left[
2(G^{+})^{\top}G
\right]
=
I_n,
$
thereby preserving the desired bilinear structure in expectation.

The trade-off is that the random matrix $(G^{+})^{\top}G$ exhibits weaker concentration than the classical covariance matrix
$G^{\top}G$. We therefore analyze $(G^{+})^{\top}G$ as a product of
sub-Gaussian random variables rather than as a quadratic form of Guassian. This weaker
concentration behavior is precisely what introduces the additional
$\gamma_1$ functional into the final projection error bound.

\section{Guarantees of random projection}

\label{sec:gurantee}

Throughout this section, we assume that strong duality holds. That is, there
exists a dual feasible multiplier \(\lambda_\star\) such that
\[
\OPT_{=} \;(\text{respectively, } \OPT_{\leq})
=
\inf_{\xv\in S}
\left\{
f(\xv)
+
\iprod{\lambda_\star}{A\xv-\bv}
\right\}.
\]
Notice that, in the inequality-constrained case, dual feasibility requires
$
\lambda_\star\geq \mathbf{0}.
$
In contrast, in the equality-constrained case, \(\lambda_\star\) is unrestricted
in sign. 

\subsection{Preservation of dual bounds in equality-constrained cases}

In this setting, we prove that random projection preserves dual bounds in equality-constrained problems. Since preservation of the primal bound follows automatically, this yields an overall guarantee for the equality-constrained case. We first establish the preservation of the dual certificate using a Gaussian decomposition and a Gaussian comparison argument.

\begin{lem}
\label{lem:prev_dual_eq}
Fix any \(\lambda\in\mathbb{R}^m\) and \(T\subseteq\mathbb{R}^m\). Let
\(G\in\mathbb{R}^{k\times m}\) have i.i.d.\ entries distributed as
\(N(0,1/k)\). Then, for every \(u\geq 1\), with probability at least
\(1-Ce^{-cu}\) for some universal constants \(c,C>0\),
\begin{align*}
\sup_{\tv\in T}
\left|
\left\langle
(G^\top G-I_m)\lambda,\tv
\right\rangle
\right|
\lesssim
\|\lambda\|_2
\left(
\frac{1}{\sqrt{k}}
+
\frac{\sqrt{u}}{k}
\right)
\left[
\gamma_2(T,\|\cdot\|_2)
+
\operatorname{rad}(T)\sqrt{u}
\right].
\end{align*}
\end{lem}

\begin{proof}
The case \(\lambda=0\) is immediate, so assume that
\(\lambda\neq 0\), and set
$
q:=\frac{\lambda}{\|\lambda\|_2}.
$
Write
$
G=\frac{1}{\sqrt{k}}\Phi,
$
where the rows of \(\Phi\) are
\(\phi_1^\top,\ldots,\phi_k^\top\), with
\(\phi_i\overset{\mathrm{i.i.d.}}{\sim}N(\mathbf 0,I_m)\).

For each \(i\), decompose \(\phi_i\) into its component along \(q\) and its
orthogonal component:
$
\phi_i=\gamma_i q+\zeta_i,
$
where
$
\gamma_i:=\langle \phi_i,q\rangle\sim N(0,1), 
\zeta_i:=P_{q^\perp}\phi_i.
$
By Gaussian orthogonal decomposition, \(\gamma_i\) and \(\zeta_i\) are
independent, and
$
\zeta_i\sim N(0,P_{q^\perp}).
$
Moreover, the pairs \((\gamma_i,\zeta_i)\) are independent across \(i\).

For any \(\tv\in T\),
\[
\begin{aligned}
\left\langle
(G^\top G-I_m)\lambda,\tv
\right\rangle
&=
\|\lambda\|_2
\left[
\frac1k\sum_{i=1}^k
\langle \phi_i,q\rangle
\langle \phi_i,\tv\rangle
-
\langle q,\tv\rangle
\right] \\
&=
\|\lambda\|_2
\left[
\frac1k\sum_{i=1}^k
\gamma_i
\langle \gamma_i q+\zeta_i,\tv\rangle
-
\langle q,\tv\rangle
\right] \\
&=
\|\lambda\|_2
\left[
a\langle q,\tv\rangle
+
\langle \ell,\tv\rangle
\right],
\end{aligned}
\]
where
$
a:=
\frac1k\sum_{i=1}^k(\gamma_i^2-1)
$
and
$
\ell:=
\frac1k\sum_{i=1}^k\gamma_i\zeta_i.
$

Consequently,
\[
\sup_{\tv\in T}
\left|
\left\langle
(G^\top G-I_m)\lambda,\tv
\right\rangle
\right|
\leq
\|\lambda\|_2
\left[
|a|\operatorname{rad}(T)
+
\sup_{\tv\in T}|\langle \ell,\tv\rangle|
\right].
\tag{2}
\label{eq:decomp_dual_cert}
\]

We control the two terms separately. Since
\(\gamma_i^2-1\) are independent centered sub-exponential random variables
with uniformly bounded \(\psi_1\)-norm, Bernstein's inequality
gives, for every \(u\geq1\),
$
|a|
\lesssim
\sqrt{\frac{u}{k}}+\frac{u}{k}
$
with probability at least \(1-2e^{-cu}\). Therefore,
\[
|a|\operatorname{rad}(T)
\lesssim
\operatorname{rad}(T)
\left(
\sqrt{\frac{u}{k}}
+
\frac{u}{k}
\right).
\tag{3}
\label{eq:a_general_bound}
\]

We next control \(\ell\). Conditional on
\(\gamma:=(\gamma_1,\ldots,\gamma_k)\), the vector
\(\sum_{i=1}^k\gamma_i\zeta_i\) is centered Gaussian with covariance
$
\left(\sum_{i=1}^k\gamma_i^2\right)P_{q^\perp}.
$
Hence, conditionally on \(\gamma\), we have
$
\ell
\stackrel{d}{=}
\frac{\|\gamma\|_2}{k}P_{q^\perp}g,
g\sim N(\mathbf 0,I_m).
$

To control \(\|\gamma\|_2\), we again apply Bernstein's inequality to the
independent centered sub-exponential random variables
\(\gamma_i^2-1\). With probability at least \(1-2e^{-cu}\),
\[
\left|
\sum_{i=1}^k(\gamma_i^2-1)
\right|
\lesssim
\sqrt{ku}+u.
\]
Therefore,
\[
\|\gamma\|_2^2
=
\sum_{i=1}^k\gamma_i^2
\lesssim
k+\sqrt{ku}+u
\lesssim
k+u,
\]
and hence
\[
\|\gamma\|_2
\lesssim
\sqrt{k}+\sqrt{u}.
\]
It follows that
\[
\frac{\|\gamma\|_2}{k}
\lesssim
\frac1{\sqrt{k}}
+
\frac{\sqrt{u}}{k}.
\]
Therefore, on this event,
\[
\sup_{\tv\in T}|\langle\ell,\tv\rangle|
\lesssim
\left(
\frac1{\sqrt{k}}
+
\frac{\sqrt{u}}{k}
\right)
\sup_{\tv\in T}
|\langle P_{q^\perp}g,\tv\rangle|.
\tag{4}
\label{eq:ell_process_bound}
\]

Consider the centered Gaussian process
\[
X_{\tv}:=\langle P_{q^\perp}g,\tv\rangle,
\qquad \tv\in T.
\]
Its canonical metric satisfies
\[
\begin{aligned}
d_X(\tv,\sv)
&:=
\left(
\mathbb E|X_{\tv}-X_{\sv}|^2
\right)^{1/2} \\
&=
\|P_{q^\perp}(\tv-\sv)\|_2
\leq
\|\tv-\sv\|_2.
\end{aligned}
\tag{5}
\]
Thus, by Lemma \ref{lem:gaussian_linear_process}, with probability at least \(1-Ce^{-cu}\),
\[
\sup_{\tv\in T}
|\langle P_{q^\perp}g,\tv\rangle|
\lesssim
\gamma_2(T,\|\cdot\|_2)
+
\operatorname{rad}(T)\sqrt{u}.
\tag{6}
\]
Combining this with \eqref{eq:ell_process_bound} yields
\[
\sup_{\tv\in T}|\langle\ell,\tv\rangle|
\lesssim
\left(
\frac1{\sqrt{k}}
+
\frac{\sqrt{u}}{k}
\right)
\left[
\gamma_2(T,\|\cdot\|_2)
+
\operatorname{rad}(T)\sqrt{u}
\right]
\tag{7}
\]
with probability at least \(1-Ce^{-cu}\).

Finally, note that
\[
\operatorname{rad}(T)
\left(
\sqrt{\frac{u}{k}}
+
\frac{u}{k}
\right)
\leq
\left(
\frac1{\sqrt{k}}
+
\frac{\sqrt{u}}{k}
\right)
\operatorname{rad}(T)\sqrt{u}.
\]
Thus, substituting \eqref{eq:a_general_bound} and \eqref{eq:ell_process_bound}
into \eqref{eq:decomp_dual_cert}, and taking a union bound over the preceding
events, we obtain
\[
\sup_{\tv\in T}
\left|
\left\langle
(G^\top G-I_m)\lambda,\tv
\right\rangle
\right|
\lesssim
\|\lambda\|_2
\left(
\frac1{\sqrt{k}}
+
\frac{\sqrt{u}}{k}
\right)
\left[
\gamma_2(T,\|\cdot\|_2)
+
\operatorname{rad}(T)\sqrt{u}
\right]
\]
with probability at least \(1-Ce^{-cu}\), for some universal constants
\(C,c>0\). This completes the proof.
\end{proof}

\begin{thm}
Let
\(G\in\mathbb{R}^{k\times m}\) have i.i.d.\ entries distributed as
\(N(0,1/k)\). Suppose that strong duality holds for the equality-constrained
problem and let \(\lambda_\star\in\mathbb{R}^m\) be an optimal dual
multiplier. Then, for every \(u\geq 1\), with probability at least
\(1-Ce^{-cu}\),
$
\OPT_{=}-\epsilon(k,u)
\leq
\ROPT_{=}(k)
\leq
\OPT_{=},
$
where
$$
\epsilon(k,u)
\lesssim
\|\lambda_\star\|_2
\left(
\frac{1}{\sqrt{k}}
+
\frac{\sqrt{u}}{k}
\right)
\left[
\gamma_2(AS,\|\cdot\|_2)
+
\bigl(
\operatorname{rad}(AS)+\|\bv\|_2
\bigr)\sqrt{u}
\right],
$$
and \(c,C>0\) are universal constants.
\end{thm}

\begin{proof}
The primal bound
\(\ROPT_{=}(k)\leq \OPT_{=}\)
follows directly from feasibility preservation. We focus on the dual side.
By strong duality, there exists an optimal dual
multiplier \(\lambda_\star\in\mathbb{R}^m\) such that
$
\OPT_{=}
=
\inf_{\xv\in S}
\left\{
f(\xv)
+
\left\langle
\lambda_\star,A\xv-\bv
\right\rangle
\right\}.
$
The randomly projected problem is
$
\ROPT_{=}(k)
=
\inf_{\xv\in S}
\left\{
f(\xv):
GA\xv=G\bv
\right\}.
$

We choose \(G\lambda_\star\) as a candidate dual solution for
\(\ROPT_{=}(k)\). By weak duality,
$$
\begin{aligned}
\ROPT_{=}(k)
&\geq
\inf_{\xv\in S}
\left\{
f(\xv)
+
\left\langle
G\lambda_\star,G(A\xv-\bv)
\right\rangle
\right\}
\\
&=
\inf_{\xv\in S}
\left\{
f(\xv)
+
\left\langle
G^\top G\lambda_\star,A\xv-\bv
\right\rangle
\right\}.
\end{aligned}
$$

Adding and subtracting \(\lambda_\star\), we obtain
$$
\begin{aligned}
\ROPT_{=}(k)
&\geq
\inf_{\xv\in S}
\Bigg\{
f(\xv)
+
\left\langle
\lambda_\star,A\xv-\bv
\right\rangle
+
\left\langle
(G^\top G-I_m)\lambda_\star,
A\xv-\bv
\right\rangle
\Bigg\}
\\
&\geq
\inf_{\xv\in S}
\left\{
f(\xv)
+
\left\langle
\lambda_\star,A\xv-\bv
\right\rangle
\right\}
-
\sup_{\xv\in S}
\left|
\left\langle
(G^\top G-I_m)\lambda_\star,
A\xv-\bv
\right\rangle
\right|
\\
&=
\OPT_{=}
-
\sup_{\xv\in S}
\left|
\left\langle
(G^\top G-I_m)\lambda_\star,
A\xv-\bv
\right\rangle
\right|.
\end{aligned}
$$

We now apply Lemma~\ref{lem:prev_dual_eq} with
$
T:=AS-\bv
=
\{A\xv-\bv:\xv\in S\}.
$
For every \(u\geq1\), with probability at least
\(1-Ce^{-cu}\),
\begin{equation}
\label{eq:form_of_eps}
\begin{aligned}
&
\sup_{\xv\in S}
\left|
\left\langle
(G^\top G-I_m)\lambda_\star,
A\xv-\bv
\right\rangle
\right|
\\
&\qquad\lesssim
\|\lambda_\star\|_2
\left(
\frac{1}{\sqrt{k}}
+
\frac{\sqrt{u}}{k}
\right)
\left[
\gamma_2(AS-\bv,\|\cdot\|_2)
+
\operatorname{rad}(AS-\bv)\sqrt{u}
\right].
\end{aligned}
\end{equation}

Since the \(\gamma_2\)-functional is translation invariant,
$
\gamma_2(AS-\bv,\|\cdot\|_2)
=
\gamma_2(AS,\|\cdot\|_2).
$
Furthermore, by the triangle inequality,
$
\operatorname{rad}(AS-\bv)
=
\sup_{\xv\in S}\|A\xv-\bv\|_2
\leq
\sup_{\xv\in S}\|A\xv\|_2+\|\bv\|_2
=
\operatorname{rad}(AS)+\|\bv\|_2.
$
Substituting these bounds into \eqref{eq:form_of_eps}, we obtain
$$
\sup_{\xv\in S}
\left|
\left\langle
(G^\top G-I_m)\lambda_\star,
A\xv-\bv
\right\rangle
\right|
\lesssim
\|\lambda_\star\|_2
\left(
\frac{1}{\sqrt{k}}
+
\frac{\sqrt{u}}{k}
\right)
\left[
\gamma_2(AS,\|\cdot\|_2)
+
\bigl(
\operatorname{rad}(AS)+\|\bv\|_2
\bigr)\sqrt{u}
\right].
$$
Therefore,
$
\OPT_{=}-\epsilon(k,u)
\leq
\ROPT_{=}(k),
$
where
$$
\epsilon(k,u)
\lesssim
\|\lambda_\star\|_2
\left(
\frac{1}{\sqrt{k}}
+
\frac{\sqrt{u}}{k}
\right)
\left[
\gamma_2(AS,\|\cdot\|_2)
+
\bigl(
\operatorname{rad}(AS)+\|\bv\|_2
\bigr)\sqrt{u}
\right].
$$
Together with
\(\ROPT_{=}(k)\leq\OPT_{=}\),
this proves the result.
\end{proof}

\subsection{Preservation of dual bounds in inequality-constrained cases}

In this setting, we prove that random projection preserves dual bounds for inequality-constrained problems. Since feasibility is not automatically preserved under random projection, this yields only a one-sided bound guarantee. 
We first establish the preservation of the dual certificate using a generic chaining of mixed-tail process.

\begin{lem}
\label{lem:prev_dual_ineq}
Fix \(\lambda\in\mathbb{R}^m\) and a nonempty set
\(T\subseteq\mathbb{R}^m\). Let
\(G\in\mathbb{R}^{k\times m}\) have i.i.d.\ entries distributed as
\(N(0,1/k)\), and let \(G^+\) denote the entrywise positive part of \(G\).
Then there exist universal constants \(c,C>0\) such that, for every
\(u\geq1\), with probability at least \(1-Ce^{-cu}\),

$$
\sup_{\tv\in T}
\left|
\left\langle
\bigl(2G^\top G^+-I_m\bigr)\lambda,\tv
\right\rangle
\right|
\lesssim
\left(\|\lambda\|_1+\|\lambda\|_2\right)
\left[
\frac{
\gamma_2(T,\|\cdot\|_2)
+
\operatorname{rad}(T)\sqrt{u}
}{\sqrt{k}}
+
\frac{
\gamma_1(T,\|\cdot\|_2)
+
\operatorname{rad}(T)u
}{k}
\right].
$$

\end{lem}

\begin{proof}
Let \(\gv_1^\top,\ldots,\gv_k^\top\) denote the rows of \(G\), so that
$
\gv_i\overset{\mathrm{i.i.d.}}{\sim}
N\left(0,\frac1k I_m\right).
$
Define the random process
$$
Z(\tv)
:=
\left\langle
\bigl(2G^\top G^+-I_m\bigr)\lambda,\tv
\right\rangle,\forall\tv\in T.
$$
Since
$
\langle G^+\lambda,G\tv\rangle
=
\sum_{i=1}^k
\langle\gv_i^+,\lambda\rangle
\langle\gv_i,\tv\rangle,
$,
we may write
$
Z(\tv)
=
\sum_{i=1}^k X_i(\tv),
$
where
$
X_i(\tv)
:=
2\langle\gv_i,\tv\rangle
\langle\gv_i^+,\lambda\rangle
-
\frac1k\langle\lambda,\tv\rangle.
$
Note that \(X_i(\tv)\) is centered. Indeed, for \(j\neq \ell\),
independence and symmetry imply
$
\mathbb{E}
\left[
(\gv_i)_j(\gv_i)_\ell^+
\right]
=0,
$
whereas
$
\mathbb{E}
\left[
(\gv_i)_j(\gv_i)_j^+
\right]
=
\mathbb{E}
\left[
(\gv_i)_j^2
\mathbf{1}_{\{(\gv_i)_j>0\}}
\right]
=
\frac{1}{2k}.
$
Therefore,
$
2\mathbb{E}
\left[
\langle\gv_i,\tv\rangle
\langle\gv_i^+,\lambda\rangle
\right]
=
\frac1k\langle\lambda,\tv\rangle,
$
and hence
$
\mathbb{E}X_i(\tv)=0 \text{ and }
\mathbb{E}Z(\tv)=0.
$

We next estimate the increments of the process. For
\(\tv,\sv\in T\),
$
X_i(\tv)-X_i(\sv)
=
2\langle\gv_i,\tv-\sv\rangle
\langle\gv_i^+,\lambda\rangle
-
\frac1k\langle\lambda,\tv-\sv\rangle.
$
We claim that
$
\left\|
X_i(\tv)-X_i(\sv)
\right\|_{\psi_1}
\lesssim
\frac{\|\lambda\|_1+\|\lambda\|_2}{k}
\|\tv-\sv\|_2.
$
Indeed, the coordinates of \(\gv_i\) are independent, centered,
sub-Gaussian random variables satisfying
$
\|G_{11}\|_{\psi_2}
\lesssim
\frac1{\sqrt{k}}.
$
Hence, the standard centered sub-Gaussian linear-form estimate gives
$
\left\|
\langle\gv_i,\tv-\sv\rangle
\right\|_{\psi_2}
\lesssim
\frac{\|\tv-\sv\|_2}{\sqrt{k}}.
$
On the other hand, the coordinates of \(\gv_i^+\) are independent,
non-centered, sub-Gaussian random variables. Since
$
\|G_{ij}^+\|_{\psi_2}
\lesssim
\|G_{ij}\|_{\psi_2}
\lesssim
\frac1{\sqrt{k}},
$
Lemma~\ref{lem:centering}, together with the standard centered
sub-Gaussian linear-form estimate, gives
$
\left\|
\langle\gv_i^+,\lambda\rangle
-
\mathbb{E}\langle\gv_i^+,\lambda\rangle
\right\|_{\psi_2}
\lesssim
\frac{\|\lambda\|_2}{\sqrt{k}}.
$
Moreover,
$
\left|
\mathbb{E}\langle\gv_i^+,\lambda\rangle
\right|
\leq
\mathbb{E}(G_{11})_+\|\lambda\|_1.
$
Since
$
\mathbb{E}(G_{11})_+
\leq
\mathbb{E}|G_{11}|
\leq
\left(\mathbb{E}G_{11}^2\right)^{1/2}
=
\frac1{\sqrt{k}},
$
we obtain
$
\left|
\mathbb{E}\langle\gv_i^+,\lambda\rangle
\right|
\leq
\frac{\|\lambda\|_1}{\sqrt{k}}.
$
Using the triangle inequality for the \(\psi_2\)-norm and
Lemma~\ref{lem:constant_psi_2}, it follows that
$
\left\|
\langle\gv_i^+,\lambda\rangle
\right\|_{\psi_2}
\lesssim
\frac{\|\lambda\|_1+\|\lambda\|_2}{\sqrt{k}}.
$
Applying Lemma~\ref{lem:product} therefore yields
$
\left\|
\langle\gv_i,\tv-\sv\rangle
\langle\gv_i^+,\lambda\rangle
\right\|_{\psi_1}
\lesssim
\frac{\|\lambda\|_1+\|\lambda\|_2}{k}
\|\tv-\sv\|_2.
$
The remaining deterministic term satisfies
$
\frac1k
\left|\langle\lambda,\tv-\sv\rangle\right|
\leq
\frac{\|\lambda\|_2}{k}\|\tv-\sv\|_2.
$
Thus, by the triangle inequality for the \(\psi_1\)-norm,
$
\left\|
X_i(\tv)-X_i(\sv)
\right\|_{\psi_1}
\lesssim
\frac{\|\lambda\|_1+\|\lambda\|_2}{k}
\|\tv-\sv\|_2.
$
Applying Lemma~\ref{lem:subexp_mixed_tail} to
$
Z(\tv)-Z(\sv)
=
\sum_{i=1}^k
\bigl(X_i(\tv)-X_i(\sv)\bigr),
$

we obtain that, for every \(v\geq1\),

$$
\mathbb{P}\left(
|Z(\tv)-Z(\sv)|
>
C(\|\lambda\|_1+\|\lambda\|_2)
\|\tv-\sv\|_2
\left(
\sqrt{\frac{v}{k}}+\frac{v}{k}
\right)
\right)
\leq
2e^{-v}.
$$

Thus, \((Z(\tv))_{\tv\in T}\) has mixed tails with respect to the
semimetrics
$
d_2(\tv,\sv)
:=
\frac{\|\lambda\|_1+\|\lambda\|_2}{\sqrt{k}}
\|\tv-\sv\|_2
$
and
$
d_1(\tv,\sv)
:=
\frac{\|\lambda\|_1+\|\lambda\|_2}{k}
\|\tv-\sv\|_2.
$

Fix an arbitrary \(\tv_0\in T\). By
Lemma~\ref{lem:mixed_tail_chaining}, for every \(u\geq1\), with
probability at least \(1-e^{-cu}\),

$$
\begin{aligned}
\sup_{\tv\in T}
|Z(\tv)-Z(\tv_0)|
\lesssim\;&
\gamma_2(T,d_2)
+
\gamma_1(T,d_1)
 +
\sqrt{u}\,\operatorname{diam}(T,d_2)
+
u\,\operatorname{diam}(T,d_1).
\end{aligned}
$$

By homogeneity of the generic chaining functionals and the semimetrics,
this gives

$$
\begin{aligned}
\sup_{\tv\in T}
|Z(\tv)-Z(\tv_0)|
\lesssim
(\|\lambda\|_1+\|\lambda\|_2)
\bigg[
&
\frac{
\gamma_2(T,\|\cdot\|_2)
+
\operatorname{diam}(T)\sqrt{u}
}{\sqrt{k}}
+
\frac{
\gamma_1(T,\|\cdot\|_2)
+
\operatorname{diam}(T)u
}{k}
\bigg].
\end{aligned}
$$

It remains to control the base point \(Z(\tv_0)\). Since
$
Z(\tv_0)=\sum_{i=1}^k X_i(\tv_0),
$
and the \(X_i(\tv_0)\) are independent, centered random variables, the same
argument used above, now with \(\sv=\mathbf 0\), gives
$
\|X_i(\tv_0)\|_{\psi_1}
\lesssim
\frac{\|\lambda\|_1+\|\lambda\|_2}{k}
\|\tv_0\|_2.
$

Therefore, Lemma~\ref{lem:subexp_mixed_tail}, or equivalently Bernstein's
inequality, implies that, with probability at least \(1-2e^{-cu}\),
$
|Z(\tv_0)|
\lesssim
(\|\lambda\|_1+\|\lambda\|_2)
\|\tv_0\|_2
\left(
\sqrt{\frac{u}{k}}
+
\frac{u}{k}
\right).
$

Since
$
\|\tv_0\|_2\leq\operatorname{rad}(T)
\text{ and }
\operatorname{diam}(T)\leq2\operatorname{rad}(T),
$ and
$
\sup_{\tv\in T}|Z(\tv)|
\leq
|Z(\tv_0)|
+
\sup_{\tv\in T}|Z(\tv)-Z(\tv_0)|
$, it
yields

$$
\sup_{\tv\in T}|Z(\tv)|
\lesssim
(\|\lambda\|_1+\|\lambda\|_2)
\left[
\frac{
\gamma_2(T,\|\cdot\|_2)
+
\operatorname{rad}(T)\sqrt{u}
}{\sqrt{k}}
+
\frac{
\gamma_1(T,\|\cdot\|_2)
+
\operatorname{rad}(T)u
}{k}
\right].
$$

A union bound over the two preceding events completes the proof.
\end{proof}

\begin{thm}
\label{thm:prev_dual_ineq}
Let
\(G\in\mathbb{R}^{k\times m}\) have i.i.d.\ entries distributed as
\(N(0,1/k)\). Suppose that strong duality holds for the
inequality-constrained problem, and let
\(\lambda_\star\in\mathbb{R}^m_+\) be an optimal dual multiplier.
Then, for every \(u\geq1\), with probability at least
\(1-Ce^{-cu}\),
$
\OPT_{\leq}-\epsilon(k)
\leq
\ROPT_{\leq}(k),
$
where
$$
\epsilon(k)
\lesssim
\left(
\|\lambda_\star\|_1+\|\lambda_\star\|_2
\right)
\left[
\frac{
\gamma_2(AS,\|\cdot\|_2)
+
\bigl(
\operatorname{rad}(AS)+\|\bv\|_2
\bigr)\sqrt{u}
}{\sqrt{k}}
+
\frac{
\gamma_1(AS,\|\cdot\|_2)
+
\bigl(
\operatorname{rad}(AS)+\|\bv\|_2
\bigr)u
}{k}
\right],
$$

and \(c,C>0\) are universal constants.
\end{thm}

\begin{proof}
By strong duality,
$
\OPT_{\leq}
=
\inf_{\xv\in S}
\left\{
f(\xv)
+
\left\langle
\lambda_\star,A\xv-\bv
\right\rangle
\right\}.
$
Since \(\lambda_\star\geq\mathbf{0}\) and \(G^+\) is entrywise
nonnegative, the vector
$
2G^+\lambda_\star\in\mathbb{R}^k_+
$
is a feasible dual multiplier for the randomly projected
inequality constraints. Hence, by weak duality,
$$
\ROPT_{\leq}(k)
\geq
\inf_{\xv\in S}
\left\{
f(\xv)
+
2\left\langle
G^+\lambda_\star,
G(A\xv-\bv)
\right\rangle
\right\}.
$$
Therefore,
\begin{align*}
\OPT_{\leq}-\ROPT_{\leq}(k)
&\leq
\sup_{\xv\in S}
\left[
\left\langle
\lambda_\star,A\xv-\bv
\right\rangle
2\left\langle
G^+\lambda_\star,G(A\xv-\bv)
\right\rangle
\right] \\
&=
\sup_{\xv\in S}
\left[
\left\langle
\lambda_\star,A\xv
\right\rangle
-
2\left\langle
G^+\lambda_\star,GA\xv
\right\rangle
\right]
+
2\left\langle
G^+\lambda_\star,G\bv
\right\rangle
-
\left\langle
\lambda_\star,\bv
\right\rangle
\\
&\leq
\sup_{\xv\in S}
\left|
\left\langle
\bigl(2G^\top G^+-I_m\bigr)\lambda_\star,
A\xv
\right\rangle
\right|
+
\left|
2\left\langle
G^+\lambda_\star,G\bv
\right\rangle
-
\left\langle
\lambda_\star,\bv
\right\rangle
\right|.
\end{align*}

We control the two terms separately. Applying
Lemma~\ref{lem:prev_dual_ineq} with \(T=AS\), we obtain, with probability
at least \(1-Ce^{-cu}\),
$$
\begin{aligned}
&\sup_{\xv\in S}
\left|
\left\langle
\bigl(2G^\top G^+-I_m\bigr)\lambda_\star,
A\xv
\right\rangle
\right|
\\
&\qquad\lesssim
\left(
\|\lambda_\star\|_1+\|\lambda_\star\|_2
\right)
\left[
\frac{
\gamma_2(AS,\|\cdot\|_2)
+
\operatorname{rad}(AS)\sqrt{u}
}{\sqrt{k}}
+
\frac{
\gamma_1(AS,\|\cdot\|_2)
+
\operatorname{rad}(AS)u
}{k}
\right].
\end{aligned}
$$
For the second term, observe that
$$
2\left\langle
G^+\lambda_\star,G\bv
\right\rangle
-
\left\langle
\lambda_\star,\bv
\right\rangle
=
\left\langle
\bigl(2G^\top G^+-I_m\bigr)\lambda_\star,
\bv
\right\rangle.
$$

Applying Lemma~\ref{lem:prev_dual_ineq} again, now with the singleton
\(T=\{\bv\}\), and noting that
$
\gamma_2(\{\bv\},\|\cdot\|_2)
=
\gamma_1(\{\bv\},\|\cdot\|_2)
=
0 \text{ and }
\operatorname{rad}(\{\bv\})
=
\|\bv\|_2,
$
gives
$$
\left|
2\left\langle
G^+\lambda_\star,G\bv
\right\rangle
-
\left\langle
\lambda_\star,\bv
\right\rangle
\right|
\lesssim
\left(
\|\lambda_\star\|_1+\|\lambda_\star\|_2
\right)
\left[
\frac{\|\bv\|_2\sqrt{u}}{\sqrt{k}}
+
\frac{\|\bv\|_2u}{k}
\right].
$$

Combining the two estimates and applying a union bound yields, with
probability at least \(1-Ce^{-cu}\),
\begin{align*}
    \OPT_{\leq}-\ROPT_{\leq}(k) \\
& \lesssim
\left(
\|\lambda_\star\|_1+\|\lambda_\star\|_2
\right)
\left[
\frac{
\gamma_2(AS,\|\cdot\|_2)
+
\bigl(
\operatorname{rad}(AS)+\|\bv\|_2
\bigr)\sqrt{u}
}{\sqrt{k}}
+
\frac{
\gamma_1(AS,\|\cdot\|_2)
+
\bigl(
\operatorname{rad}(AS)+\|\bv\|_2
\bigr)u
}{k}
\right].
\end{align*}
This proves the desired one-sided bound.
\end{proof}

\begin{rem}
\label{rem:lin_help}
Suppose that the objective function \(f(\cdot)\) is linear. Then the objective appearing in the weak-duality bound is linear
in \(\xv\). Hence, the infimum may be restricted to the vertices of \(S\):
\begin{align*}
\ROPT_{\leq}(k)
&\geq
\inf_{\xv\in S}
\left\{
f(\xv)
+
2\left\langle
G^+\lambda_\star,
G(A\xv-\bv)
\right\rangle
\right\}
=
\inf_{\xv\in V(S)}
\left\{
f(\xv)
+
2\left\langle
G^+\lambda_\star,
G(A\xv-\bv)
\right\rangle
\right\},
\end{align*}
where \(V(S)\) denotes the set of vertices of \(S\).

Consequently, the corresponding random process can be indexed by
\(AV(S)\) and the error bound can be sharpened to
$$
\epsilon(k,u)
\lesssim
\left(
\|\lambda_\star\|_1+\|\lambda_\star\|_2
\right)
\left[
\frac{
\gamma_2(AV(S),\|\cdot\|_2)
+
\bigl(
\operatorname{rad}(AS)+\|\bv\|_2
\bigr)\sqrt{u}
}{\sqrt{k}}
+
\frac{
\gamma_1(AV(S),\|\cdot\|_2)
+
\bigl(
\operatorname{rad}(AS)+\|\bv\|_2
\bigr)u
}{k}
\right].
$$
In general,
$
\gamma_2(AV(S),\|\cdot\|_2)
\asymp
\gamma_2(AS,\|\cdot\|_2),
$
whereas
$
\gamma_1(AV(S),\|\cdot\|_2)
\leq
\gamma_1(AS,\|\cdot\|_2).
$
Thus, linearity of the objective function does not substantially improve the \(\gamma_2\)-term, but it may yield a strictly smaller \(\gamma_1\)-term; see Table~\ref{tab:gamma_simplex_ball}.
\end{rem}

\subsection{Simultaneous preservation of primal-dual bounds in inequality-constrained cases}

In general, classical random projection cannot simultaneously preserve both
the primal and dual bounds for inequality-constrained problems. A simple
example is when the feasible set \(S\) is a singleton. Suppose
\(S=\{\bar{\xv}\}\) and that the original problem is strictly feasible, so that
$
A\bar{\xv}<\bv.
$
Let
$
\rv:=\bv-A\bar{\xv}\in\mathbb{R}^m_{+},
\rv\neq 0.
$
Then the projected problem is feasible if and only if
$
GA\bar{\xv}\leq G\bv,
$
or equivalently,
$
G\rv\geq 0.
$
If \(G\in\mathbb{R}^{k\times m}\) has i.i.d.\ centered Gaussian entries, then
each coordinate of \(G\rv\) is an independent centered Gaussian random
variable. Hence,
$
\mathbb{P}(G\rv\geq 0)=2^{-k}.
$
Therefore,
$
\mathbb{P}\bigl(
\text{the projected problem is infeasible}
\bigr)
=
1-2^{-k}.
$
Thus, even though the original problem is feasible, its randomly projected
counterpart is infeasible with overwhelming probability as \(k\) grows.

Therefore, in order to simultaneously preserve primal and dual bounds using
the classical signed Gaussian projection, one needs additional
assumptions on both the feasible region and the objective function. Such
assumptions must account for the rotational nature of Gaussian projections.
For instance, one may quantify how ``ball-like'' the feasible region is, so
that rotating a feasible solution is unlikely to destroy feasibility, together
with how sensitive the objective value is to such rotations. An alternative way to recover primal feasibility is to modify the projection
matrix itself. For example, \cite{poirion2023random} considers a nonnegative
random projection obtained from the Hadamard square (entrywise square) of a
Gaussian matrix. Since the resulting projection matrix is nonnegative,
$
A\xv\leq\bv
\Longrightarrow
GA\xv\leq G\bv,
$
and hence the projected problem remains a relaxation of the original problem.
The price of this construction, however, is that it no longer enjoys the same
unbiased bilinear-preservation property underlying the classical Gaussian
dual certificate. In particular, the natural dual projection used in the
equality-constrained Gaussian setting does not directly yield an analogous
preservation of the dual bound.

If one gives up the requirement of a vanishing projection error and allows a
controlled bias in the projected constraints, then a simultaneous primal--dual
bound can be recovered. Consider
\begin{equation}
\label{eq:sm_bound}
    \begin{aligned}
         \ROPT_{\leq}(k,\alpha,\beta) := \inf_{\xv} \quad & f(\xv) \\
        \text{s.t.} \quad
        &  GA\xv
     \leq
     \alpha G\bv
     +
     \frac{\beta}{\sqrt{k}}\ev, \\
        & \xv \in S.
    \end{aligned}
\end{equation}
Suppose that \(\mathbf{0}\in S\), \(S\) is convex, \(f\) is convex, and
\(f(\mathbf{0})=0\). The assumptions \(\mathbf{0}\in S\) and
\(f(\mathbf{0})=0\) can be imposed essentially without loss of generality by
translating the feasible region and shifting the objective function by a
constant.
Let \(\xv_\star\) be an optimal solution of the original problem and define
$
    \xv_\alpha:=\alpha\xv_\star.
$
Then \(\xv_\alpha\in S\), and by convexity,
$
f(\xv_\alpha)
=
f(\alpha\xv_\star+(1-\alpha)\mathbf{0})
\leq
\alpha f(\xv_\star)
=
\alpha\OPT_{\leq}.
$
Moreover, writing
$
    \rv_\star:=\bv-A\xv_\star\geq0,
$
the feasibility of \(\xv_\alpha\) for~\eqref{eq:sm_bound} reduces to
controlling the negative fluctuations of the Gaussian vector
\(G\rv_\star\), since
$
    GA\xv_\alpha-\alpha G\bv
    =
    -\alpha G\rv_\star.
$
Each coordinate of \(G\rv_\star\) is centered Gaussian with standard
deviation of order \(\|\rv_\star\|_2/\sqrt{k}\). Therefore, choosing
\(\beta\) of order
$
   \alpha\|\rv_\star\|_2\sqrt{\log k}
$
ensures, with high probability, that \(\xv_\alpha\) is feasible. Hence,
$
    \ROPT_{\leq}(k,\alpha,\beta)
    \leq
    \alpha\OPT_{\leq}.
$

On the other hand, the dual argument is essentially the same as in
Theorem~\ref{thm:prev_dual_ineq}. Using \(2G^+\lambda_\star\) as a feasible
dual multiplier for the projected problem yields, with high probability,
$
    \OPT_{\leq}
    -
    \epsilon(k,\alpha,\beta)
    \leq
    \ROPT_{\leq}(k,\alpha,\beta),
$
where \(\epsilon(k,\alpha,\beta)\) consists of the previous random-projection
error together with the additional bias introduced by the modified
right-hand side. In particular, the perturbation
\(\beta\ev/\sqrt{k}\) contributes the term
$
    \frac{2\beta}{\sqrt{k}}
    \left\langle G^+\lambda_\star,\ev\right\rangle,
$
whose contribution is of order
$
    O\!\left(
        \beta\|\lambda_\star\|_1
        +
        \beta\|\lambda_\star\|_2
        \sqrt{\frac{u}{k}}
    \right)
$
with high probability. Therefore, the dominant additional error is of order
\(\beta\|\lambda_\star\|_1\), which does not vanish as \(k\to\infty\) when
\(\beta\) is bounded away from zero. 

Combining the primal and dual arguments gives a simultaneous bound of the
form
$$
    \OPT_{\leq}
    -
    \epsilon(k,\alpha,\beta)
    \leq
    \ROPT_{\leq}(k,\alpha,\beta)
    \leq
    \alpha\OPT_{\leq},
$$
where, schematically,
$
    \epsilon(k,\alpha,\beta)
    =
    \text{vanishing error}
    +
    O\!\left(
        \beta\|\lambda_\star\|_1\right)
$
Thus, unlike the equality-constrained case, simultaneous preservation of the
primal and dual bounds is possible only after introducing a controlled,
generally nonvanishing bias. Since the proof follows the same arguments as
above, we omit the details.

\section{Conclusion}

In this paper, we study random projection in optimization through Lagrangian
duality and generic chaining. For equality-constrained problems, this viewpoint
yields projection guarantees for general convex feasible regions, with the
error governed by the geometric complexity
\(\gamma_2(AS,\|\cdot\|_2)\), and can lead to sharper bounds than
simplex-based analyses. For inequality-constrained problems, signed Gaussian
projection does not preserve primal feasibility. Nevertheless, we show that
the dual bound can still be approximately preserved by projecting the dual
multiplier as \(2G^+\lambda_\star\). The resulting mixed-tail process leads
naturally to error bounds involving both \(\gamma_2\) and \(\gamma_1\).

Several questions remain open. Most notably, it would be interesting to
characterize when simultaneous preservation of primal and dual bounds with
vanishing error is possible for inequality-constrained problems. This may
require additional geometric assumptions on the feasible region and objective
function, or alternative random projections that balance preservation of
primal feasibility with preservation of dual certificates.

\appendix

\section{Remaining proofs}

\label{app_proofs}
\textbf{Proof of Lemma \ref{lem:gaussian_linear_process}}

\begin{proof}
Consider the centered Gaussian process
$
    X_{\tv}:=\langle Pg,\tv\rangle,
    \tv\in T.
$
Its canonical metric satisfies
$
\begin{aligned}
d_X(\tv,\sv)
=
\left(
\mathbb{E}|X_{\tv}-X_{\sv}|^2
\right)^{1/2} 
=
\|P^\top(\tv-\sv)\|_2
\leq
\|\tv-\sv\|_2.
\end{aligned}
$
By Gaussian comparison Lemma \cite{vershynin2018highdim}
$
    \mathbb{E}
    \sup_{\tv\in T}
    |X_{\tv}|
    \lesssim
    \gamma_2(T,\|\cdot\|_2)
    +
    \operatorname{rad}(T).
$
Next, since the map
$
    g\mapsto
    \sup_{\tv\in T}
    |\langle Pg,\tv\rangle|
$
is \(\operatorname{rad}(T)\)-Lipschitz, 
Hence, Gaussian concentration  \cite{vershynin2018highdim} implies that, with probability at least
\(1-Ce^{-cu}\),

$$
    \sup_{\tv\in T}
    |\langle Pg,\tv\rangle|
    \lesssim
    \gamma_2(T,\|\cdot\|_2)
    +
    \operatorname{rad}(T)
    +
    \operatorname{rad}(T)\sqrt{u}.
$$

Since \(u\geq1\), the middle term can be absorbed into
\(\operatorname{rad}(T)\sqrt{u}\), which proves the claim.
\end{proof}

\begin{prop}
\label{prop_gamma_1}
Let
$
\Delta_n
=
\operatorname{conv}\{0,e_1,\ldots,e_n\}.
$
Then
$
\gamma_1(\Delta_n,\|\cdot\|_2)
\asymp
\sqrt n.
$
\end{prop}

\begin{proof}
Since
$
\Delta_n\subseteq B_1^n,
$
it suffices for the upper bound to estimate
\(\gamma_1(B_1^n,\|\cdot\|_2)\).
Let
$
e_r(B_1^n,\|\cdot\|_2)
:=
\inf\left\{
\varepsilon>0:
N(B_1^n,\|\cdot\|_2,\varepsilon)
\leq 2^{r-1}
\right\}
$
denote the entropy numbers of \(B_1^n\) under the Euclidean metric.
By Sch\"utt~\cite{schutt1984entropy},
\[
e_r
\asymp
\begin{cases}
1, & r\lesssim \log n,\\[1mm]
\displaystyle
\sqrt{\frac{\log(1+n/r)}{r}},
& \log n\lesssim r\lesssim n,\\[3mm]
\displaystyle
n^{-1/2}2^{-r/n},
& r\gtrsim n.
\end{cases}
\]
Using the entropy bound \cite{van2014probability} for the chaining functional,
$
\gamma_1(T,d)
\lesssim
\sum_{j\geq0}2^j e_{2^j}(T,d),
$
we obtain, since \(\Delta_n\subseteq B_1^n\),
$$
\gamma_1(\Delta_n,\|\cdot\|_2)
\leq
\gamma_1(B_1^n,\|\cdot\|_2)
\lesssim
\sum_{j\geq0}
2^j e_{2^j}(B_1^n,\|\cdot\|_2).
$$
We split the sum according to the three regimes of the entropy numbers.

First, when \(2^j\lesssim \log n\), we have
\(e_{2^j}\asymp 1\), and hence
\[
\sum_{2^j\lesssim\log n}
2^j e_{2^j}
\lesssim
\sum_{2^j\lesssim\log n}2^j
\lesssim
\log n
\lesssim
\sqrt n.
\]

Next, suppose
$
\log n\lesssim 2^j\lesssim n.
$
Then
$
e_{2^j}
\asymp
\sqrt{
\frac{\log(1+n/2^j)}{2^j}
},
$
so that
$
2^j e_{2^j}
\asymp
\sqrt{
2^j\log(1+n/2^j)
}.
$
Let \(J\) satisfy \(2^J\asymp n\) and write \(j=J-\ell\). Then
$
2^j\asymp \frac{n}{2^\ell},
\log\left(1+\frac{n}{2^j}\right)
\lesssim
\ell+1.
$
Therefore,
\begin{align*}
\sum_{\log n\lesssim2^j\lesssim n}
2^j e_{2^j}
&\lesssim
\sqrt n
\sum_{\ell\geq0}
\sqrt{\frac{\ell+1}{2^\ell}} \lesssim
\sqrt n,
\end{align*}
since
$
\sum_{\ell\geq0}
\sqrt{\frac{\ell+1}{2^\ell}}
<\infty.
$

Finally, when \(2^j\gtrsim n\),
$
e_{2^j}
\asymp
n^{-1/2}2^{-2^j/n}.
$
Thus,
$
2^j e_{2^j}
\asymp
\sqrt n
\left(\frac{2^j}{n}\right)
2^{-2^j/n}.
$
Again writing \(2^j/n\asymp 2^\ell\), we obtain
\[
\sum_{2^j\gtrsim n}
2^j e_{2^j}
\lesssim
\sqrt n
\sum_{\ell\geq0}
2^\ell 2^{-c2^\ell}
\lesssim
\sqrt n.
\]
Combining the three regimes gives
$
\gamma_1(\Delta_n,\|\cdot\|_2)
\lesssim
\sqrt n.
$

For the reverse inequality, we use the standard entropy lower bound \cite{van2014probability}
\[
\gamma_1(T,d)
\gtrsim
\sup_{\varepsilon>0}
\varepsilon\log N(T,d,\varepsilon).
\]
Thus, it suffices to exhibit a scale \(\varepsilon\) at which
\(\Delta_n\) has exponentially large covering number.

To this end, by the Varshamov-Gilbert bound \cite{tsybakov2009introduction}, there exists a set
\(\mathcal C\subseteq\{0,1\}^n\) such that
$
|\mathcal C|\geq e^{cn}, 
d_H(\sigma,\tau)\geq cn
$
for all distinct \(\sigma,\tau\in\mathcal C\). For each
\(\sigma\in\mathcal C\) and some constant $c$, define
$
x^\sigma:=\frac{\sigma}{n}.
$
Since
$
\sum_{i=1}^n x_i^\sigma
=
\frac{\|\sigma\|_1}{n}
\leq 1,
$
we have \(x^\sigma\in\Delta_n\). Moreover, for distinct
\(\sigma,\tau\in\mathcal C\),
$
\|x^\sigma-x^\tau\|_2
=
\frac{\sqrt{d_H(\sigma,\tau)}}{n}
\gtrsim
\frac{1}{\sqrt n}.
$
Hence \(\Delta_n\) contains an exponentially large packing at scale
\(1/\sqrt n\), and therefore
$
\log N\left(
\Delta_n,\|\cdot\|_2,\frac{c}{\sqrt n}
\right)
\gtrsim n.
$
Applying the entropy lower bound at
\(\varepsilon\asymp n^{-1/2}\) gives
$
\gamma_1(\Delta_n,\|\cdot\|_2)
\gtrsim
\frac{1}{\sqrt n}\,n
=
\sqrt n.
$
Combining this with the upper bound yields
$
\gamma_1(\Delta_n,\|\cdot\|_2)
\asymp
\sqrt n.
$
\end{proof}

\addcontentsline{toc}{chapter}{References}
\bibliographystyle{plain}
\bibliography{ref}

\appendix

\end{document}